\documentclass[11pt]{amsart}

\usepackage{amssymb}

\newtheorem{thm}{Theorem}[section]

\newtheorem{lem}[thm]{Lemma}
\newtheorem{prop}[thm]{Proposition}
\newtheorem*{t5}{Theorem 5.1}

\theoremstyle{remark}
\newtheorem{ntn}[thm]{Notation}
\newtheorem{rem}[thm]{Remark}

\newcommand{\pl}{$\Gamma = ({\mathcal P}, {\mathcal L})$}
\newcommand{\mss}[2]{M_{#1}{(#2)}}
\newcommand{\res}[2]{Res_{#1}{(#2)}}

\begin{document}

\title[On a space related to the building of type ${\widetilde E}_7$]{On a space related to the affine building of type $\mathbf{E_7}$}

\author{Silvia Onofrei}
\address{Department of Mathematics, Kansas State University, 137 Cardwell Hall, Manhattan, Kansas 66506, Email: onofrei@math.ksu.edu}

\begin{abstract}A locally truncated geometry with diagram of type affine $E_7$ is studied. One considers a parapolar space $\widetilde{\Gamma}$, locally of type $A_{7,4}$, which is subject to an extra axiom. A covering $\Gamma$ of this space is constructed. It is proved that $\Gamma$ is a rank 6, residually connected, locally truncated diagram geometry which is a homomorphic image of a truncated building of affine type $E_7$. Consequently, $\widetilde{\Gamma}$ is also a homomorphic image of a truncated building.
\end{abstract}

\maketitle

\section{Introduction}
The purpose of this work is to study a parapolar space which is locally of type $A_{7,4}$ and which is subject to an extra assumption called the Weak Hexagon Axiom. The present results complete a more general analysis of those parapolar spaces which are locally of type $A_{n-1, 4}({\mathbb K})$, with $n$ a suitably chosen integer and ${\mathbb K}$ a field, performed by the author and described below. To simplify the notation we will omit the field ${\mathbb K}$.

\medskip
In a previous paper \cite{on1}, we studied parapolar spaces ${\widetilde \Gamma}$ which were locally of type $A_{n-1,4}$ (with $n > 6$ and $ n \not =  8$) and which satisfied:

\medskip
{\it The Weak Hexagon Axiom (WHA): Let $H = (p_1, \ldots p_6)$ be a
$6$-circuit, isometrically embedded in ${\widetilde \Gamma}$, this means that $p_i \in p^{\perp}_{i+1}$, indices taken mod $6$, and all the other pairs are not collinear. Also assume that at least one of the pairs of points at distance two, say $\lbrace p_1, p_3 \rbrace$, is polar. Then there exists a point $w \in p_1^{\perp} \cap p_3^{\perp} \cap p_5^{\perp}$.}

\medskip
First it was proved that ${\widetilde \Gamma}$ can be enriched to a locally truncated geometry of rank $5$. Using the fact that ${\widetilde \Gamma}$ inherits from the local structure two classes of maximal singular subspaces, a new family of objects ${\widetilde {\mathcal D}}$ was constructed. The elements of ${\widetilde {\mathcal D}}$ are $2$-convex subspaces of type $D_{6,6}$; see \cite[Theorem 2]{on1}. Second, using a sheaf theoretic argument it was shown that ${\widetilde \Gamma}$ was a homomorphic image of a truncated building; see \cite[Theorem 3.b]{on1}.

\medskip
The case with $n=7$ which was not considered in \cite{on1} is the subject of this work. The difficulty of this case relies on the fact that all the maximal singular subspaces have the same singular rank, hence a partition of the maximal singular subspaces according to dimension is no longer attainable and therefore many of the arguments based on the local properties cannot be used. In order to overcome this difficulty, we shall construct a locally truncated diagram geometry {\pl} which is a covering of ${\widetilde \Gamma}$ and which is locally isomorphic to ${\widetilde \Gamma}$, thus of local type $A_{7,4}$, satisfies {\it (WHA)} and in addition, has the property that the maximal singular subspaces can be separated in two families which we shall denote $\mathcal{A}$ and $\mathcal{B}$.

\medskip
From this point our approach is similar to the one used in \cite{on1}. We start by constructing two new collections of $2$-convex subspaces ${\mathcal D}_{\mathcal A}$ and  ${\mathcal D}_{\mathcal B}$ whose elements are of type $D_{6,6}$. We prove that every symplecton $S \in \mathcal{S}$ of $\Gamma$ is contained in exactly one element of ${\mathcal D}_{\mathcal A}$ and one element of ${\mathcal D}_{\mathcal B}$. Therefore the space $\Gamma$ can be enriched to a rank six geometry $(\mathcal{P, L, A, B, D_A, D_B})$. Then we construct a sheaf over this locally truncated diagram geometry and we prove that $\Gamma$ is a homomorphic image of a truncated building of affine type $E_7$.

\medskip
In Section $2$ we provide the reader with a list of definitions and basic results which will be used in the latter sections. In Section $3$, we use a theorem by Kasikova and Shult to construct the covering $\Gamma$. The main result of Section $4$ is contained in Theorem $4.1$ which attests the existence of two families of subspaces ${\mathcal D}_{\mathcal A}$ and ${\mathcal D}_{\mathcal B}$ and that every symplecton of $\Gamma$ lies in exactly two subspaces, one from each family. In Section $5$, we prove the main result of the paper:

\begin{t5}Let ${\widetilde \Gamma} = ({\widetilde {\mathcal P}}, {\widetilde {\mathcal L}})$ be a parapolar space which is locally of type $A_{7, 4}$. Let $I = \lbrace 1, \ldots 8 \rbrace, \; J = \lbrace 1, 8 \rbrace$ and $K = I \setminus J$. Assume that ${\widetilde \Gamma}$ satisfies the Weak Hexagon Axiom. Then there is a residually connected $J$-locally truncated diagram geometry $\Gamma$ belonging to the diagram:\\
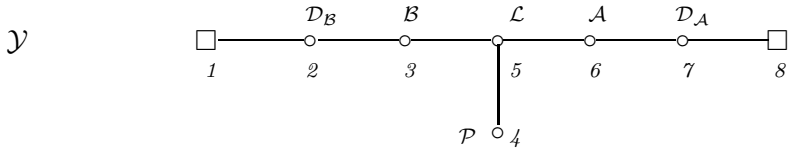
\begin{picture}(1000, 50)(0,0)
\put(0,25){$\mathcal{Y}$}
\put(60,25){$\qed$}
\put(75,15){\scriptsize 1}
\put(80,29){\line(1,0){32}}
\put(112,26){$\circ$}
\put(113,15){\scriptsize 2}
\put(113,36){\scriptsize ${\mathcal D}_{\mathcal B}$}
\put(118,29){\line(1,0){30}}
\put(148,26){$\circ$}
\put(150,36){\scriptsize ${\mathcal B}$}
\put(150,15){\scriptsize 3}
\put(153,29){\line(1,0){30}}
\put(183,26){$\circ$}
\put(190,36){\scriptsize ${\mathcal L}$}
\put(190,15){\scriptsize 5}
\put(186,-3){\line(0,1){30}}
\put(183,-9){$\circ$}
\put(171,-10){\scriptsize ${\mathcal P}$}
\put(190,-10){\scriptsize 4}
\put(188,29){\line(1,0){30}}
\put(218,26){$\circ$}
\put(220,36){\scriptsize ${\mathcal A}$}
\put(220,15){\scriptsize 6}
\put(223,29){\line(1,0){30}}
\put(253,26){$\circ$}
\put(253,36){\scriptsize ${\mathcal D}_{\mathcal A}$}
\put(255,15){\scriptsize 7}
\put(258,29){\line(1,0){30}}
\put(276,25){$\qed$}
\put(290,15){\scriptsize 8}
\end{picture}

\vspace{.4cm}
whose universal covering is the truncation of a building. Therefore, ${\widetilde \Gamma}$ is also the homomorphic image of a truncated building.
\end{t5}

\section{Preliminaries and definitions}
\subsection{Geometries} In this section the basic definitions related to geometries are given; for an expository treatment see \cite[Chapter 3]{hnbk}.

\medskip
A {\it geometry over} $I$ is a system $\Gamma = (V, *, t)$ consisting of set $V$, a binary, symmetric, reflexive relation on $V$ and a mapping $t: V \rightarrow I$. The elements of $V$ are called {\it objects}, $*$ is called the {\it incidence relation} and $t$ is the {\it type function} of $\Gamma$.

\medskip
A {\it flag} $F$ of $\Gamma$ is a (possibly empty) subset of pairwise incident objects of $\Gamma$. The set $t(F)$ is the {\it type} of $F$ and the set $I \setminus t(F)$ is its {\it cotype}. The cardinalities of these sets are the {\it rank} and the {\it corank} of $F$. The {\it residue} of $F$ in $\Gamma$ is the geometry $\res{\Gamma}{F} = (V_F, *_{|V_F}, t _{|V_F})$ over $I \setminus t(F)$, where $V_F$ is the set of all members of $V \setminus F$ incident with each element of $F$. The corank of a flag $F$ is the rank of $\res {\Gamma}{F}$. A geometry $\Gamma$ is {\it residually connected} if and only if for every flag $F$ of corank at least one, $\res{\Gamma}{F}$ is not empty and if, for each flag of corank at least two, $\res{\Gamma}{F}$ is connected.

\medskip
Let $\Gamma _k = (V_k, *_k, t_k)$ with $k \in K$ some index set and where each $\Gamma_k$ is a geometry over $I_k$. Assume that $\lbrace I_k \rbrace _{k \in K}$ is a family of pairwise disjoint sets. The {\it direct sum of geometries} is the geometry denoted by $\Gamma = \bigoplus _{k \in K} \Gamma _k = (V, *, t)$, where $V = \bigcup _{k \in K} V_k$. The incidence is defined as follows: $*_{|V_k}\;  :=  *_k$ and $x * y$ for any two objects $x \in V_{k_1}$ and $y \in V_{k_2}$ with $k_1 \not = k_2$. Finally $t _{|V_k}\;  := t_k$.

\medskip
Let ${\mathcal Geom}_I$ denote the category whose objects are
geometries with typeset $I$ and whose morphisms are the type preserving graph morphisms. A {\it fibering morphism of geometries} $\phi: \Gamma_1 \rightarrow \Gamma_2$ is an object surjective morphism which is object bijective when restricted to the residue of each object in $\Gamma_1$.

\subsection{Locally truncated geometries}
We gather in this subsection the necessary notions on locally truncated geometries needed in this paper; for a more detailed account of the concepts from this subsection the reader is referred to Brouwer and Cohen \cite{loc} and Ronan \cite{ron}.

\medskip
Let $I$ be an index set and let $J \subset I$. The {\it truncation of type $J$} of $\Gamma$, denoted by $^J \Gamma$ is the geometry obtained by restricting the typeset of
$\Gamma$ to $J$. The truncation is a functor $^J {\bf Tr}: {\mathcal Geom}_I \rightarrow {\mathcal Geom}_J$ from the category of geometries over $I$
to the category of geometries over $J$.

\medskip
The {\it $J$-truncation of $\Gamma$}, denoted by $_J \Gamma$, has as objects the objects of $\Gamma$ whose types
are in $I \setminus J$, incidence and type function are those from
$\Gamma$ but restricted to $I \setminus J$. Differently said, the
$J$-truncation of $\Gamma$ is the truncation of type $I \setminus J$ of $\Gamma$, that is $_J \Gamma = \; ^{I \setminus J}\Gamma$.

\medskip
A {\it diagram} $D$ over $I$ is a mapping which assigns to each $2$-subset $\lbrace i, j \rbrace$ of $I$, a class $D(i,j)$ of rank $2$ geometries.  A geometry $\Delta$ over $I$ {\it belongs to the diagram} $D$ if and only if every
residue of type $\lbrace i, j \rbrace$ of $\Delta$ is a geometry from $D(i,j)$.

\medskip
A geometry $\Gamma$ over $I \setminus J$ is said to be $J$-{\it locally truncated of type $D$} (or {\it with diagram $D$}) over
$I$ if and only if for every nonempty flag $F$ of $\Gamma$, the residue $\res{\Gamma}{F}$ is isomorphic to the truncation of type $I \setminus ( J \cup t(F))$ of a geometry belonging to the diagram $D_{I \setminus t(F)}$, the restriction of $D$
to the typeset $I \setminus t(F)$. If $\Gamma$ is the truncation of type $I \setminus J$ of a geometry
$\Delta$ of type $D$ over $I$ then $\Gamma$ it is a geometry of $J$-locally
truncated type $D$. The converse is in general not true; see Brouwer and Cohen
\cite{loc} and Ronan \cite{ron}.

\medskip
Let $M=(m_{ij})$ be a Coxeter matrix with rows and columns indexed by $I$. The {\it
diagram} of $M$, denoted $D(M)$, assigns to each $2$-subset $\lbrace i, j \rbrace$ of $I$, the class $D(i,j)$ of generalized $m_{i j }$-gons. If $\Gamma$ is a residually connected geometry with diagram $D(M)$, that is every residue of type $\lbrace i, j \rbrace$ of $\Gamma$ is a generalized $m_{i j}$-gon, then $\Gamma$ is called a {\it geometry of type} $M$.

\subsection{Chamber systems}
The notion of chamber system was introduced by Tits \cite{tits}. We give below basic definitions and results which will be used in Section $5$; for a detailed account on chamber systems see \cite{cp}.

\medskip
A {\it chamber system} ${\mathcal C} = (C, E, \lambda, I)$ over $I$ is a simple graph $(C, E)$ together with an edge-labeling $\lambda: E \rightarrow 2^{I} \setminus \lbrace \emptyset \rbrace$ by nonempty subsets of $I$ such that, if $a, b, c \in C$ are three pairwise adjacent vertices, then $\lambda (a,b) \cap \lambda (b,
c) \subseteq \lambda (a,c)$. The elements of $C$ are called
{\it chambers}. Two distinct chambers $a$ and $b$ are $i$-{\it adjacent} iff $(a,b) \in E$ is an edge and $i \in \lambda (a,b)$, for any $i \in I$. The rank of ${\mathcal C}$ is the cardinality of the index set $I$. The chamber systems over $I$ together with the appropriate morphisms form a category denoted ${\mathcal Chamb}_I$.

\medskip
For $J$ a subset of $I$, the {\it residue of ${\mathcal C}$ of type $J$} or the $J$-{\it residue}, is a connected component of the graph $(C, E_J, \lambda _J, J)$, with
$\lambda_J$ the restriction of $\lambda$ to $\lambda^{-1}(2^J) \subseteq E$, where $2^J$ is the codomain of $\lambda _J$, and $E_J =\lbrace e \in E \mid \lambda _J (e) \not=
\emptyset \rbrace $. A $J$-residue $R$ is a chamber system over the typeset $J$. The set $I \setminus J$ is called the {\it cotype} of $R$. The {\it rank of the residue} $R$ is $| J|$ and its {\it corank} is $| I \setminus J |$.

\medskip
A chamber system over $I$ is {\it residually connected} if and only if for every subset $J \subseteq I$ and for every family $\lbrace R_j: j\in J \rbrace$ of residues of cotype $j$, with the property that any two have nonempty intersection, it follows that $\cap _{j \in J}R_j$ is a nonempty residue of type $I \setminus J$.

\medskip
Let $\Gamma$ be a geometry over $I$. Denote by ${\bf C}(\Gamma)$ the set of its chamber flags, that is, the flags of type $I$. Two chamber flags $c$ and $d$ are said to be $i$-adjacent whenever they have the same element of type $j$ for all $j \not = i$, with $i, j \in I$. Then ${\bf C}(\Gamma)$, with the above adjacency relation, is a chamber system of type $I$. Starting with a chamber system ${\mathcal C}$ over $I$, we define a geometry ${\bf G}({\mathcal C}) = (C_i, i\in I, *, t)$. The objects of this geometry are the elements $C_i, i \in I$, the collection of all corank one residues of type $I \setminus \lbrace i \rbrace$ of ${\mathcal C}$. Two objects are incident if they have nonempty intersection. The above construction gives rise to a pair of functors ${\mathbf G}: {\mathcal Chamb}_I \rightarrow
{\mathcal Geom}_I$ and
${\mathbf C}: {\mathcal Geom}_I \rightarrow {\mathcal Chamb}_I$ such that ${\mathbf G}( {\bf C}(\Gamma)) = \Gamma$, if $\Gamma$ is a residually connected geometry and $I$ is finite, and ${\mathbf C}({\mathbf G}({\mathcal C}))= {\mathcal C}$, if ${\mathcal C}$ is a residually connected chamber system. For more details see \cite{asp}.

\medskip
For a connected chamber system ${\mathcal C}$ over $I$, a {\it $2$-cover} of ${\mathcal C}$ is a
connected chamber system $\widetilde {\mathcal C}$ together with a chamber system
morphism $h : \widetilde {\mathcal C} \rightarrow {\mathcal C}$ which is surjective on chambers and is an isomorphism when restricted to any residues of rank at most $2$ of $\widetilde {\mathcal C}$. A $2$-cover $h: \widetilde {\mathcal C} \rightarrow {\mathcal C}$ is
said to be {\it universal} if for any other $2$-cover $\varphi : {\mathcal C}'
\rightarrow {\mathcal C}$ there is a $2$-cover $\psi : \widetilde {\mathcal C} \rightarrow
{\mathcal C}'$ such that $h \circ \psi = \varphi$. It can be proved that
chamber systems always have universal coverings.

\medskip
A {\it chamber system ${\mathcal C}$ over $I$ belongs to the
diagram $D(M)$}, with $M$ a Coxeter matrix, if and only if every residue of ${\mathcal C}$ of type $\lbrace i, j \rbrace \subseteq I$ is the chamber system of a generalized
$m_{ij}$-gon; one also says that ${\mathcal C}$ is a chamber system of type $M$. {\it Buildings} are chamber systems of type $M$ which satisfy extra axioms. For a complete definition of the buildings see \cite{tits}. The following result, known as Tits' Local Approach Theorem, can be found in \cite[Corollary 3]{tits}:

\begin{thm}[Tits]Suppose ${\mathcal C}$ is a chamber system of type $M$ with $M$ a Coxeter
matrix, and suppose that for every rank $3$ residue, the universal $2$-cover is a building. Then the universal $2$-cover of ${\mathcal C}$ is a building ${\mathcal B}$ of type $M$.
\end{thm}

In particular, the chamber system ${\mathcal C}$ is obtained from ${\mathcal B}$ by factoring out a group
of automorphisms in which no non-trivial element fixes any rank $2$
residue of ${\mathcal B}$.

\subsection{Sheaves}
Let $I$ be an index set, $J \subset I$ and set $K = I \setminus J$. Let $\Gamma$ be a geometry over $K$ which is locally truncated of type $D$ over $I$ and let ${\mathcal F}$ be a family of nonempty flags of $\Gamma$. A {\it sheaf over the geometry} $\Gamma$ is a class of geometries $\lbrace \Sigma(F) \; \text{for} \; F \in {\mathcal F} \rbrace$ together with isomorphisms
\begin{center}
$\varphi_F: \res{\Gamma}{F} \rightarrow \; _J \Sigma (F)$
\end{center}
of geometries over $I \setminus t(F)$. Given a pair of incident flags $F_1 \subseteq F_2$ in ${\mathcal F}$ the connecting homomorphisms of the sheaf are the maps $\varphi _{F_1, F_2} : \Sigma (F_2) \rightarrow \Sigma (F_1)$ with the property that
\begin{center} $\varphi_{F_1, F_2}(\Sigma(F_2)) \simeq \res{\Sigma(F_1)}{F_2 \setminus F_1}$.
\end{center}
Furthermore, they are subject to the following conditions:
\begin{center}
$\varphi _{F_1, F_2} \circ \varphi _{F_2, F_3} = \varphi _{F_1, F_3}$,
\end{center}
for $F_1, F_2, F_3 \in {\mathcal F}$ with $F_1 \subseteq F_2 \subseteq F_3$. To simplify the notation, we will omit the connecting isomorphisms $\varphi _F$ and write $\res{\Gamma}{F} = \; _J \Sigma(F)$ instead.

\medskip
A sheaf $\Sigma$ is {\it residually connected} if and only if for each
object $x$ of the geometry $\Gamma$ the sheaf geometry $\Sigma (x)$ is
residually connected. Due to the functorial relation between the category of geometries and the
category of chamber systems, whenever a sheaf $\Sigma$ exists, there
is a chamber system associated to it \cite[Lemma 1]{loc}.

\subsection{Spaces}
This section is devoted to the basic concepts used in Section $3$; for a good review on spaces and related topics the reader is referred to \cite[Chapter 12]{hnbk}.

\medskip
A {\it space} $\Gamma$ is a pair $({\mathcal P}, {\mathcal L})$ consisting of a nonempty set ${\mathcal P}$, whose members are called {\it points}, and a collection ${\mathcal L}$ of subsets of ${\mathcal P}$ of cardinality at least two, whose members are called {\it lines}. Let $p, q \in {\mathcal P}$ be two distinct points. We say that $p$ {\it is collinear to} $q$ if both $p$ and $q$ lie in some line $L \in {\mathcal L}$. The set of all points of ${\mathcal P}$ collinear with $p$, including $p$ itself, will be denoted $p^{\perp}$; it is called {\it the perp of} $p$. A space is a {\it partial linear space} if two distinct points lie in at most one line. A space is a {\it gamma space} if it is a partial linear space and for any $p \in {\mathcal P}$ and $L \in {\mathcal L}$ the set $p^{\perp} \cap L$ is empty, a point or $L$.

\medskip
The {\it collinearity graph} of a space is the graph whose vertex set is ${\mathcal P}$ and in which two points are adjacent if they are distinct and collinear. Given two points $p, q \in {\mathcal P}$, the distance between $p$ and $q$ in the collinearity graph will be denoted $d(p,q)$.

\medskip
A subset $X$ of the point set ${\mathcal P}$ is a {\it subspace} of $\Gamma$ if every line $L \in {\mathcal L}$ meeting $X$ in at least two points entirely belongs to $X$. A subspace is {\it singular} if every two points of $X$ are collinear. A subspace is $2$-{\it convex} if for all pairs of points $p, q \in X$ with $d(p,q)=2$, each point which is collinear with both $p$ and $q$ is also contained in $X$.

\medskip
The {\it singular rank} of a space $\Gamma$ is the length of the longest chain of distinct nonempty singular subspaces $X_0 \subset X_1 \subset \ldots \subset X_n$. In what follows the term {\it rank} will denote the projective rank of a singular subspace of $\Gamma$, that is one less than the vector space dimension. A polar space is said to be of rank $n$ if its singular rank is $n-1$.

\medskip
A {\it parapolar space} is a connected partial linear gamma space possessing a collection of $2$-convex subspaces $\mathcal S$, called {\it symplecta}, isomorphic to nondegenerate polar spaces of rank at least $2$, with the properties that each line is
contained in a symplecton and each quadrangle is contained in a unique symplecton. A parapolar space in which every pair of points at distance $2$ belongs to a symplecton, is also called a {\it strong parapolar space}. Let $p, q \in {\mathcal P}$ be two points at distance two in a parapolar space. If $|p ^{\perp} \cap q^{\perp}|=1$ then $\lbrace p, q \rbrace$ is called a {\it special pair}; if $|p^{\perp} \cap q^{\perp}|>1$ then $\lbrace p, q \rbrace$ is a {\it polar pair}. If $\lbrace p, q \rbrace$ is a polar pair, then the convex closure of $p$ and $q$ is the unique symplecton containing the two points which will be denoted $\ll p, q \gg$. In a parapolar space all the singular subspaces are projective spaces. In a parapolar space {\pl}, the {\it residue of a point} $p \in {\mathcal P}$ is the space ${\rm Res}_{\Gamma}p = ({\mathcal L}_p, \pi_p)$ induced on the lines and the planes which contain $p$.

\medskip
Let ${\mathfrak D}_n$ be a Coxeter diagram with $n$ nodes and select a node $i$ in the diagram. A space is said to be of type ${\mathfrak D}_{n,i}$ if it is the shadow space over $i$ of a building of type ${\mathfrak D}_n$; see \cite[Chapter 12, Section 4.7]{hnbk}.

\subsubsection{The space of type $A_{7,4}$}
Let $V$ be an $8$-dimensional vector space over some division ring $\mathbb{K}$. Define the space {\pl}  whose points $\mathcal{P}$ are the $4$-subspaces of $V$ and whose lines $\mathcal{L}$ are the $(3,5)$-flags of $V$. Then $\Gamma$ is a strong parapolar space whose symplecta are polar spaces of type $D_{3,1}$. Let $ \mathcal S$ denote the family of symplecta. The maximal singular subspaces ${\mathcal M}$ have rank $4$ and can be partitioned into two classes ${\mathcal A}$ and ${\mathcal B}$, according to the property {\it(G1)} below. We list some of the properties of the space of type $A_{7,4}$; these properties can be derived by linear algebra arguments. A characterization of the spaces of type $A_{n,j}$ can be found in \cite{coh}.

\begin{itemize}
\item[{\it(G1)}] If $M_1,M_2 \in {\mathcal M}$ are two distinct maximal singular
subspaces belonging to the same class then $  M_1 \cap  M_2$ is either empty or a point. If they belong to different classes then $  M_1 \cap   M_2$ is either empty or a line.
\item[{\it(G2)}] If $p \in {\mathcal P}$, $S \in {\mathcal S}$, $p \not\in S$ then $ p^{\perp} \cap S$ is empty, a point or a plane.
\item[{\it(G3)}] If $S \in {\mathcal S}$ and $M \in {\mathcal M}$ then $S \cap M$ is empty, a point or a plane.
\item[{\it(G4)}] If $( p,M) \in {\mathcal P} \times {\mathcal M}$ and $ p \not\in  M$ then $p^{\perp} \cap  M$ is either empty or a line.
\end{itemize}

\subsubsection{The space of type $D_{6,6}$}
This is the space {\pl} whose points ${\mathcal P}$ are one class of maximal singular subspaces of the polar space of type $D_{6,1}$. The lines ${\mathcal L}$ are the rank $3$ singular subspaces of the same polar space. Then $\Gamma$ is a strong parapolar space, with a family ${\mathcal S}$ of symplecta which are polar spaces of type $D_{4,1}$. Let ${\mathcal S}$ denote the collection of symplecta. There are two classes of maximal singular subspaces ${\mathcal A}$, whose elements
have rank $5$ and ${\mathcal B}$, a class of rank $3$ singular subspaces.
The diagram of a space of type $D_{6,6}$ is given below:\\

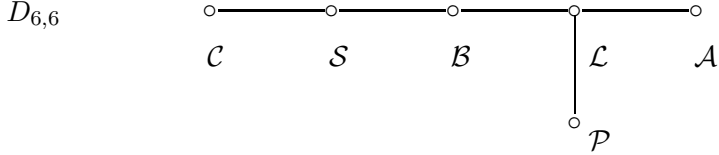
\begin{picture}(1000, 50)(0,0)
\put(0,25){$D_{6,6}$}
\put(74,26){$\circ$}
\put(76,8){${\mathcal C}$}
\put(80,29){\line(1,0){40}}
\put(120,26){$\circ$}
\put(122,8){${\mathcal S}$}
\put(126,29){\line(1,0){40}}
\put(166,26){$\circ$}
\put(168,8){${\mathcal B}$}
\put(172,29){\line(1,0){40}}
\put(212,26){$\circ$}
\put(220,8){${\mathcal L}$}
\put(215,-10){\line(0,1){37}}
\put(212,-16){$\circ$}
\put(220,-24){${\mathcal P}$}
\put(218,29){\line(1,0){40}}
\put(258,26){$\circ$}
\put(260,8){${\mathcal A}$}
\end{picture}

\vspace*{1cm}
The space of type $D_{6,6}$ was initially characterized by Cohen and Cooperstein \cite[Theorem 4]{cc}:

\begin{thm}[Cohen and Cooperstein]Let {\pl} be a strong parapolar space of singular rank $5$, which is not a polar space and whose symplecta have rank $4$. Assume that, given a point-symplecton pair $(p, S) \in {\mathcal P} \times {\mathcal S}$ with $p \not \in S$, the set $p^{\perp} \cap S$ is either a point or a maximal singular subspace of $S$. Then $\Gamma$ is a space of type $D_{6,6}$.\\
\end{thm}

\section{The maximal singular subspaces}

Suppose that ${\widetilde \Gamma} = ({\widetilde{\mathcal P}}, {\widetilde{\mathcal L}})$ is parapolar space which is locally of type $A_{2n-1,n}$ with $n$ a positive integer.

\smallskip
The Grassmann space $A_{2n-1, n}(\mathbb{K})$ and its quotient $A_{2n-1, n}(\mathbb{K})/\langle \sigma \rangle$, where $\mathbb{K}$ is an infinite division ring and where $\sigma$ is an involutory automorphism of $A_{2n-1,n}(\mathbb{K})$ induced by a polarity of the underlying projective space of Witt index at most $n-4$, have the same diagram. In this case ${\widetilde \Gamma}$ can have point residuals of both types $A_{2n-1, n}(\mathbb{K})$ and $A_{2n-1, n}(\mathbb{K})/ \langle \sigma \rangle$. In the second geometry the maximal singular subspaces are fused in one single family; see \cite[Section 6]{coh} for example. Since the case of interest here is $n=4$, the Grassmann space of type $A_{7,4}$ is too ``small" and it does not have nontrivial quotients which are parapolar spaces, thus a local partition of the maximal singular subspaces into two classes can be attained (see Section 2.5.1 for details).

\medskip
However, all the maximal singular subspaces of ${\widetilde \Gamma}$ have the same rank and there is no global partition in two classes according to the rank. In order to overcome this difficulty, we shall use a result of Kasikova and Shult \cite[Section 3.4]{ks} and we shall construct a covering of ${\widetilde \Gamma}$ in which the partition of the maximal singular subspaces in two global classes can be realized. We adapt Theorem $12$ from \cite{ks} to the case of interest here and for completeness we also provide a proof following \cite{ks}.

\begin{thm}[Kasikova and Shult] Let ${\widetilde \Gamma} = ({\widetilde {\mathcal P}}, {\widetilde {\mathcal L}})$ be a parapolar space which is locally of type $A_{7,4}$. Let $I = \lbrace 1, \ldots 8 \rbrace$ and $J = \lbrace 1,2,7,8 \rbrace$.  Then there is a $J$-locally truncated connected geometry $\Gamma$ with diagram:\\
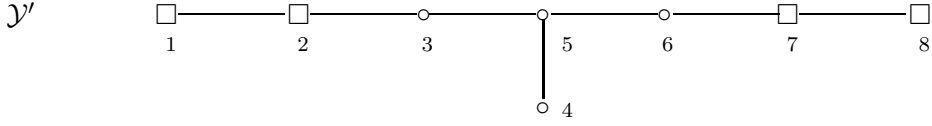
\begin{picture}(1000, 50)(0,0)
\put(0,25){$ \mathcal{Y'}$}
\put(45,25){$\qed$}
\put(60,15){\scriptsize $1$}
\put(65,29){\line(1,0){40}}
\put(95,25){$\qed$}
\put(110,15){\scriptsize ${2}$}
\put(115,29){\line(1,0){40}}
\put(155,26){$\circ$}
\put(157,15){\scriptsize ${3}$}
\put(160,29){\line(1,0){40}}
\put(200,26){$\circ$}
\put(210,15){\scriptsize $5$}
\put(203,-3){\line(0,1){30}}
\put(200,-9){$\circ$}
\put(210,-10){\scriptsize $4$}
\put(206,29){\line(1,0){40}}
\put(246,26){$\circ$}
\put(248,15){\scriptsize $6$}
\put(252,29){\line(1,0){40}}
\put(280,25){$\qed$}
\put(295,15){\scriptsize $7$}
\put(300,29){\line(1,0){40}}
\put(330,25){$\qed$}
\put(345,15){\scriptsize $8$}
\end{picture}

\vspace{.5cm}
and a fibering morphism of geometries $\phi_{\Gamma}: \Gamma \rightarrow {\widetilde \Gamma}$ which induces a one- or two-fold $\mathcal {T}$-covering $\varphi _{\Gamma}: \Delta _{\Gamma} \rightarrow \Delta_{\widetilde \Gamma}$ of the point-collinearity graph of ${\widetilde \Gamma}$.
\end{thm}

\begin{proof} Let ${\widetilde \Gamma} = ({\widetilde{\mathcal P}}, {\widetilde{\mathcal L}})$ be a parapolar space which is locally of type $A_{7,4}$. Then for each point ${\widetilde p} \in {\widetilde{\mathcal P}}$ the residue ${\rm Res}_{\widetilde \Gamma}{\widetilde p}$ contains two classes of maximal singular subspaces ${\widetilde{\mathcal A}}_{\widetilde p}$ and ${\widetilde {\mathcal B}}_{\widetilde p}$ whose elements are projective spaces of rank $5$. Let ${\widetilde {\mathcal M}}_{\widetilde p} = {\widetilde {\mathcal A}}_{\widetilde p} \cup {\widetilde {\mathcal B}}_{\widetilde p}$ and set ${\widetilde {\mathcal M}} = \lbrace {\widetilde {\mathcal M}} _{\widetilde p} \mid {\widetilde p} \in {\widetilde {\mathcal P}} \rbrace$. Each plane ${\widetilde \pi}$ on ${\widetilde p}$ lies in exactly one element from each class. Let ${\mathbb P}$ denote the set of pairs $({\widetilde p}, {\widetilde{\mathcal X}}_{\widetilde p})$ where ${\widetilde p} \in {\widetilde{\mathcal P}}$ and ${\widetilde{\mathcal X}}_{\widetilde p}$ is one of the classes ${\widetilde{\mathcal A}}_{\widetilde p}$ or ${\widetilde{\mathcal B}}_{\widetilde p}$. Let ${\mathbb L}$ denote the collection of pairs $({\widetilde L}, {\widetilde{\mathcal X}}_{\widetilde L})$ with ${\widetilde L} \in {\widetilde {\mathcal L}}$ and ${\widetilde{\mathcal X}}_{\widetilde L}$ a class of maximal singular subspaces containing ${\widetilde L}$. We say that $({\widetilde p}, {\widetilde{\mathcal X}}_{\widetilde p})$ is incident with $({\widetilde L}, {\widetilde{\mathcal X}}_{\widetilde L})$ if and only if ${\widetilde p}$ is incident with ${\widetilde L}$ in ${\widetilde \Gamma}$ and ${\widetilde{\mathcal X}}_{\widetilde L} \subseteq {\widetilde{\mathcal X}}_{\widetilde p}$. Thus $\Upsilon = ({\mathbb P},{\mathbb L})$ is a point-line geometry.

\medskip
The geometry morphism $\phi: \Upsilon \rightarrow {\widetilde \Gamma}$ induced by the projection onto the first coordinate is vertex surjective. Let now $({\widetilde p}, {\widetilde{\mathcal X}}_{\widetilde p}) \in {\mathbb P}$ and $({\widetilde L}_i, {\widetilde{\mathcal X}}_{{\widetilde L}_i}) \in {\mathbb L}$ with $i=1,2$, be two lines incident with $({\widetilde p}, {\widetilde{\mathcal X}}_{\widetilde p})$. Observe that $\phi(({\widetilde L}_1, {\widetilde{\mathcal X}}_{{\widetilde L}_1})) = \phi(({\widetilde L}_2, {\widetilde{\mathcal X}}_{{\widetilde L}_2}))$ implies that ${\widetilde L}_1 = {\widetilde L}_2$. Since $({\widetilde p}, {\widetilde{\mathcal X}}_{\widetilde p})$ is incident with both lines we must have  ${\widetilde {\mathcal X}}_{{\widetilde L}_1} = {\widetilde{\mathcal X}}_{{\widetilde L}_2}$ and therefore $({\widetilde L}_1, {\widetilde{\mathcal X}}_{{\widetilde L}_1}) = ({\widetilde L}_2, {\widetilde{\mathcal L}}_{{\widetilde L}_2})$. A similar argument can be applied if we start with a line in ${\mathbb L}$ and two points in ${\mathbb P}$ incident with it. Thus $\phi$ is one-to-one when restricted to neighbor graphs, the set of all the lines (or points) which are incident with a fixed point (or line respectively).

\medskip
The map $\phi$ also induces a vertex surjective morphism on the point-collinearity graphs: $\varphi: \Delta_{\Upsilon} \rightarrow \Delta_{\widetilde \Gamma}$ which restricts to an isomorphism $({\widetilde p}, {\widetilde{\mathcal X}}_{\widetilde p})^{\perp_{\Upsilon}} \rightarrow {\widetilde p}^{\perp_{\widetilde{\Gamma}}}$. Furthermore, all $3$-circuits in $\Delta_{\widetilde \Gamma}$ lift to $3$-circuits in $\Delta_{\Upsilon}$. In the language of \cite[Section 2]{ks}, the restriction $\varphi _{\Gamma}$ of $\varphi$ to a connected component of $\Delta_{\Upsilon}$ is a $\mathcal{T}$-covering.

\medskip
In $\Upsilon$ there are two global classes of maximal singular subspaces. They can be defined as follows: for any ${\widetilde M} \in {\widetilde {\mathcal M}}$ let $\mathcal{A}_{\widetilde M} = \lbrace ({\widetilde p}, {\widetilde{\mathcal X}}_{\widetilde p}({\widetilde M}))| {\widetilde p} \in {\widetilde M} \rbrace$ where ${\widetilde {\mathcal X}}_{\widetilde p}({\widetilde M})$ denotes the class ${\widetilde{\mathcal A}}_{\widetilde p}$ or ${\widetilde{\mathcal B}}_{\widetilde p}$ which contains ${\widetilde M}$ and $\mathcal{B}_{\widetilde M} = \lbrace ({\widetilde p}, {\widetilde{\mathcal X}'}_{\widetilde p}({\widetilde M}))| {\widetilde p} \in {\widetilde M} \rbrace$ where ${\widetilde {\mathcal X}'}_{\widetilde p}({\widetilde M})$ denotes the class which does not contain ${\widetilde M}$. Then the two families of maximal singular subspaces are ${\mathcal A} = \lbrace {\mathcal A}_{\widetilde M} | {\widetilde M} \in {\widetilde {\mathcal M}} \rbrace$ and $\mathcal{B} = \lbrace {\widetilde {\mathcal B}}_{\widetilde M} | {\widetilde M} \in {\widetilde {\mathcal M}} \rbrace$.

\medskip
It is quite clear that $\Upsilon$ can be either connected (when the local classes of maximal singular subspaces of ${\widetilde \Gamma}$ are fused) or disconnected (when a global separation is possible) and has two connected components. Let $\Gamma$ be a connected component of $\Upsilon$. This is mapped by $\phi$ as either a one-to-one mapping or a two-to-one mapping depending on whether $\Gamma$ is a proper subgeometry or not. So if we denote by $\phi_{\Gamma}$ the restriction of $\phi$ to $\Gamma$ we obtain the morphism from the conclusion of the theorem.
\end{proof}

\begin{rem} The point-line geometry $\Gamma = ({\mathcal P}, {\mathcal L})$ is connected and locally isomorphic to ${\widetilde \Gamma}$, which means that $\Gamma$ is locally of type $A_{7,4}$, hence it belongs to diagram $\mathcal{Y}'$. The maximal singular subspaces of $\Gamma$ partition in two classes which we shall denote ${\mathcal A}$ and ${\mathcal B}$. Then $\Gamma$ is a parapolar space with a class of symplecta ${\mathcal S}$ of type $D_{4,1}$.
\end{rem}

\begin{rem} Assume ${\widetilde \Gamma}$ satisfies the Weak Hexagon Axiom {\it (WHA)}. It is a consequence of the definition of the fibering morphism of geometries that ({\it WHA}) is also valid in $\Gamma$. Therefore, in our analysis we can replace the space ${\widetilde \Gamma}$ with its covering $\Gamma$, gaining in this way the advantage of being able to partition the maximal singular subspaces into two classes.
\end{rem}

\section{The family ${\mathcal D}$ of subspaces}

Throughout this section {\pl} will denote a parapolar space which is locally of type $A_{7,4}$ and it is constructed as in the proof of Theorem $3.1$ of the previous section. The local assumption means that if $p \in {\mathcal P}$ then ${\rm Res}_{\Gamma} p$ is a space of type $A_{7,4}$. Recall the basic notions on spaces and related structures from Section 2.5. The result of the previous section allows to obtain a global partition of the maximal singular subspaces into two classes; see {\it (P1)} below for notations and details. Next, $\Gamma$ contains a set of symplecta ${\mathcal S}$ whose elements are (non-degenerate) polar spaces of type $D_{4,1}$.

\medskip
Let $S$ be a symplecton in $\Gamma$ and let ${\mathcal X} \in \lbrace {\mathcal A}, {\mathcal B} \rbrace$. We define the following two sets associated to $S$:

$$N_{\mathcal X}(S)= \lbrace p \in {\mathcal P} \setminus S \mid
p^{\perp}
\cap S \in M_{\mathcal X}(S) \rbrace$$
and
\begin{equation*}
\begin{split}
G_{\mathcal X}(S) & = \lbrace p \in {\mathcal P} \setminus S \mid
p^{\perp} \cap S =
\lbrace q \rbrace, \text{ a point};\; \text{for any}\; r \in q^{\perp}\cap S\\
&\quad \text{the pair} \; \lbrace p,r \rbrace \; \text{is polar}; \; \text{for some} \; A \in {\mathcal A}_q \cap {\mathcal A}(S), \; p^{\perp} \cap A \; \text{is a plane} \rbrace.
\end{split}
\end{equation*}

\smallskip
Then set
$$D_{\mathcal X}(S)= S \cup N_{\mathcal X}(S) \cup G_{\mathcal X}(S)$$
and let ${\mathcal D}_{\mathcal X}= \lbrace D_{\mathcal X}(S)|\; S \in {\mathcal S} \rbrace$. Finally denote ${\mathcal D} = {\mathcal D}_{\mathcal A} \cup {\mathcal D}_{\mathcal B}$.

\medskip
The main result of this section is:

\begin{thm} Let {\pl} be a parapolar space  which is locally of type $A_{7,4}$ and which satisfies:

\smallskip
The Weak Hexagon Axiom {\it (WHA)}: Let $H = (p_1, \ldots p_6)$ be a
$6$-circuit, isometrically embedded in $\Gamma$, this means that $p_i \in p^{\perp}_{i+1}$, indices taken mod $6$, and all the other pairs are not collinear. Also assume that at least one of the pairs of points at distance two, say $\lbrace p_1, p_3 \rbrace$, is polar. Then there exists a point $w \in p_1^{\perp} \cap p_3^{\perp} \cap p_5^{\perp}$.

\smallskip
Then there exist two collections of $2$-convex subspaces ${\mathcal D}_{\mathcal A}$ and ${\mathcal D}_{\mathcal B}$, whose elements are spaces of type $D_{6,6}$. Every symplecton $S \in {\mathcal S}$ lies in exactly one element from each class.
\end{thm}

In order to prove the theorem we first show that each $D_{\mathcal X}$
is a $2$-convex subspace which is a strong parapolar space in its own. Then we use the Cohen-Cooperstein characterization theorem, see Section 2.5.2, to conclude that $D_{\mathcal X}$ is of type $D_{6,6}$.

\medskip
The following properties of $\Gamma$ are consequences of the theory of parapolar spaces and of the local assumption; see also the properties {\it (G1-G4)} given in Section 2.5.1.

\medskip
{\it (P1)} There are two classes of maximal singular subspaces
${\mathcal A}$ and ${\mathcal B}$. The set ${\mathcal M} = {\mathcal
A} \cup
{\mathcal B}$ is a collection of singular subspaces of rank $5$. Two distinct maximal singular
subspaces which belong to the same class can meet at a line, a
point or the empty set. Two maximal singular subspaces from different
classes can meet at a plane, a point or they are disjoint.\\
{\it (P2)} If $(S, M) \in {\mathcal S} \times {\mathcal M}$ then the set $S
\cap M$ can be empty, a point, a line  or a
maximal singular subspace of $S$.\\
{\it (P3)} For $S \in {\mathcal S}$ let ${\mathcal  M}(S)$ be the family of maximal singular subspaces with largest intersection with $S$; observe that ${\mathcal M}(S) = {\mathcal A}(S) \cup {\mathcal B}(S)$. For $M_i \in {\mathcal M}(S),\; i=1,2$, denote ${\overline M}_i = M_i \cap S$. Then ${\overline M}_1 \cap {\overline M}_2$ can be a
point or a plane if ${\overline M}_1, {\overline M}_2$ belong to different classes; ${\overline M}_1 \cap {\overline M}_2$ can be empty or a line when ${\overline M}_1$ and ${\overline M}_2$ are disjoint and belong to the same class.\\
{\it (P4)} If $M \in {\mathcal M}$ and $p \in
{\mathcal P} \setminus  M$ then
$p^{\perp} \cap M$ can be empty, a point or a plane.\\
{\it (P5)} If $S \in {\mathcal S}$ and $p \in {\mathcal P} \setminus S$
then $p^{\perp} \cap S$ can be empty, a point, a line or a maximal singular subspace of
$S$.

\begin{ntn} For $S \in {\mathcal S}$ and ${\mathcal X} \in \lbrace {\mathcal A}, {\mathcal B}
\rbrace$ let $M_{\mathcal X}(S)= \lbrace {\overline X} \mid {\overline X} = X \cap  S
\; \; \text{for some} \; X \in {\mathcal X}(S) \rbrace$ be the two classes of maximal singular subspaces of $S$.
\end{ntn}

\begin{ntn} In what follows if $p \in {\mathcal
P}$ and $S \in {\mathcal S}$ are such that  $p ^{\perp} \cap S \in M_{\mathcal X}(S)$,
then we denote ${\overline X}_p:= p^{\perp} \cap S$ and the maximal singular
subspace containing it $X_p$. If $p, q \in {\mathcal P}$ form a polar pair in $\Gamma$, the unique symplecton containing them will be denoted $\ll p, q \gg$.
\end{ntn}

\begin{lem}
Let $S$ be a symplecton in $\Gamma$.\\
a). If $p \in N_{\mathcal X}(S)$ and $q \in p^{\perp}$ is such that
$q^{\perp} \cap S \setminus {\overline X}_p \not= \emptyset$, then $q \in S \cup N_{\mathcal X}(S)$.\\
b). The set $S \cup N_{\mathcal X}(S)$ is a subspace of $\Gamma$.
\end{lem}

\begin{proof} a). As the statement is obviously true when $q \in S$, one may assume that $q \not\in S$. Let $r \in q^{\perp} \cap S \setminus {\overline X}_p$. Set $R= \ll r,p \gg$ and note that $q \in R$. Hence $L:=q^{\perp} \cap r^{\perp} \cap \overline X_p$ is a line and $q^{\perp} \cap S$ contains the plane $\langle r, L \rangle $. According to $\it{(P5)}$, $q^{\perp} \cap S = {\overline M}$ is a maximal singular subspace of $S$. Now $\langle qr, {\overline M}\cap {\overline X}_p\rangle$ has rank $3$ and lies in $R$, thus it is a maximal singular subspace of $R$. It follows that ${\overline M}\cap {\overline X}_p=L$ and according to $\it{(P3)}$, ${\overline M}$ and ${\overline X}_p$ belong to the same class. Hence $q \in N_{\mathcal X}(S)$.

\medskip
b). One has to show that if $p$ and $q$ are two collinear points in $S \cup N_{\mathcal X}(S)$, then the line $pq$ lies entirely in $S \cup N_{\mathcal X}(S)$. If at least one of the points $p$ or $q$ is in $S$ the conclusion follows at
once. Assume $p, q \not\in S$. Set ${\overline X}_p = p^{\perp} \cap S$ and ${\overline X}_q = q^{\perp} \cap S$, two
maximal singular subspaces of $S$ belonging to the same class. According to ${\it (P3)}$
there are three cases to consider:

\smallskip
$(i)$. ${\overline X}_p = {\overline X}_q$ in which case $X_p = X_q$ and
thus $pq \subset N_{\mathcal X}(S)$.

\smallskip
$(ii)$. ${\overline X}_p \cap {\overline X}_q = L$ is a line in $S$.
Let $r \in pq \setminus \lbrace p, q \rbrace$. Clearly $L \subset r^{\perp}$. Let $w \in {\overline X}_q \setminus r^{\perp}$. Hence ${\overline N}= \langle w, w^{\perp} \cap {\overline X}_p \rangle$ is a maximal singular subspace of $S$, not in the
same class with ${\overline X}_p$ and ${\overline X}_q$ since ${\overline N} \cap {\overline X}_p$ and ${\overline N} \cap {\overline
X}_q$ are both planes in $S$. Set $R =  \ll w,p \gg$ and observe that $\overline N \subseteq R \cap S$. In $R$, the set $r^{\perp} \cap {\overline N}$ is a plane and by ${\it(P5)}$, $r ^{\perp}\cap S = {\overline X}_r$ is a maximal singular subspace of $S$. It remains to show that ${\overline X}_r \in M_{\mathcal X}(S)$. It suffices to prove: ${\overline X}_p \cap {\overline X}_r =L$. Now ${\overline X}_r$ meets ${\overline N}$ at the plane $r^{\perp} \cap {\overline N}$. So, the family of ${\overline X}_r$ is not the same as that of ${\overline N}$, which in its turn is not the same as that of ${\overline X}_p$ and ${\overline X}_q$. Hence ${\overline X}_r$ and ${\overline X}_p$, being different, meet at a line, which must be $L$.

\smallskip
$(iii)$. The singular subspaces ${\overline X}_p$ and ${\overline X}_q$ are disjoint. Let $t \in {\overline X}_p$ and set $T =  \ll t, q \gg$. Then the plane $t^{\perp} \cap {\overline X}_q$ lies in $T$ and therefore $p^{\perp} \cap t^{\perp} \cap {\overline X}_q$ is a line. But this contradicts the assumption that ${\overline X}_p \cap {\overline X}_q = \emptyset$. Thus
either ${\overline X}_p \cap {\overline X}_q \not= \emptyset$ or $p$ is not collinear with $q$.
\end{proof}

\begin{lem} Let $p \in G_{\mathcal X}(S)$ with $\lbrace q \rbrace = p^{\perp} \cap S$.\\
a). For any $X \in {\mathcal X}_q \cap {\mathcal X}(S)$ the set $p^{\perp} \cap X$ is a plane.\\
b). Let $r \in N_{\mathcal X}(S)$ be such that $q \in r^{\perp} \cap S$. Then either $r \in p^{\perp}$ or $\lbrace p, r \rbrace$ is a polar pair.
\end{lem}

\begin{proof}
a). The result is a consequence of the local structure. For a proof see \cite[Lemma 5.1]{on1}.\\
b). Follows directly from part a) of the lemma.
\end{proof}

In what follows, for ease of exposition, we shall let ${\mathcal X}={\mathcal A}$ and show that $D_{\mathcal A}(S)$ is a space of type $D_{6,6}$. Afterwards, we can replace ${\mathcal A}$ with ${\mathcal B}$ and repeat the arguments to obtain that $D_{\mathcal B}(S)$ is also a space of type $D_{6,6}$.

\begin{lem} Let $R, S \in {\mathcal S}$. If $R \cap S= \overline B$ is a maximal
singular subspace in $M_{\mathcal B}(S) \cap M_{\mathcal B}(R)$ then $D_{\mathcal A}(S) = D_{\mathcal A}(R)$.
\end{lem}

\begin{proof} It suffices to prove $D_{\mathcal A}(S) \subseteq D_{\mathcal A}(R)$.

\medskip
$(i)$. Let $p \in S \setminus {\overline  B}$. Then $p ^{\perp} \cap \overline B$ is a plane and, by {\it (P5)}, it follows that $p^{\perp} \cap R = \overline A_p$ is a maximal singular subspace of $R$ in $M_{\mathcal A}(R)$ (since it has a plane in common with $\overline B$). Thus $p \in N_{\mathcal A}(R) \subset D_{\mathcal A}(R)$ and $S \subset D_{\mathcal A}(R)$.

\medskip
$(ii)$. Let now $p \in N_{\mathcal A}(S)$. If $p \in R$ we are done so we may assume that $p \not \in R$. Then $\overline A_p = p^{\perp} \cap S$ can have either a plane or a point in common with $\overline B$. If $\overline A_p \cap \overline B$ is a plane, then $p \in N_{\mathcal A}(R)$ by a similar argument to that used in $(i)$. Next let $\overline A_p \cap \overline B = \lbrace q \rbrace$ be a single point.
If there exists a point $r \in p^{\perp} \cap R \setminus \lbrace q \rbrace$ then $r \not \in A_p$. Assume to the contrary that $r \in A_p$. Then $r ^{\perp} \cap S \supseteq \langle {\overline A}_p, r^{\perp} \cap {\overline B} \rangle \supset {\overline A}_p$ and since ${\overline A}_p$ is a maximal singular subspace of $S$ we obtain a contradiction with {\it (P5)}. Thus $ r \not \in A_p$. Let $t \in {\overline A}_p \setminus r^{\perp}$, which, according to Step $(i)$, is in $N_{\mathcal A}(R)$. An application of Lemma $4.4(a)$, to the pair $\lbrace p, t \rbrace$ and symplecton $R$, gives that $p \in N_{\mathcal A}(S)$. Next assume that $p^{\perp} \cap R = \lbrace q \rbrace$, a single point. {\it Claim}: $p \in G_{\mathcal A}(R)$. Let $s \in q^{\perp} \cap R \setminus S$ so $s \in R \subset N_{\mathcal A}(S)$ and the pair $\lbrace p, s \rbrace$ is polar. Let ${\overline A}_s = s^{\perp} \cap S$ with $A_s$ be the maximal singular subspace containing ${\overline A}_s$. Since $p^{\perp} \cap A_s$ contains the line ${\overline A}_p \cap {\overline A}_s$, it follows, by {\it (P4)}, that $p^{\perp} \cap A_s$ is a plane. This concludes the proof of the claim.

\medskip
$(iii)$. Let $p \in G_{\mathcal A}(S)$ be such that $p^{\perp} \cap {\overline B} = \lbrace q \rbrace$ is a point. Assume first that $p^{\perp} \cap R = \lbrace q \rbrace$. Let $r \in q^{\perp} \cap R$, so $r \in N_{\mathcal A}(S)$. By Lemma $4.5(b)$ the pair $\lbrace p, r  \rbrace$ is polar. Further, if $A \in {\mathcal A}(S) \cap {\mathcal A}(R) \cap {\mathcal A}_q$ then $p^{\perp} \cap A$ is a plane. Thus $p \in G_{\mathcal A}(R)$. Assume now $p^{\perp} \cap R \supseteq qs$, a line. Consider ${\rm Res}_{\Gamma}(q)$. For a subspace  $F$ of $\Gamma$ which contains $q$, ${\widetilde F}$ will denote the corresponding subspace in $\res{\Gamma}{q}$. In $\res{\Gamma}{q}$, ${\widetilde R}$ and ${\widetilde S}$ are two ``symplecta" which meet at a plane of $\widetilde {\mathcal B}$-type. Also $\tilde p$ is a ``point" at distance two from every single ``point" in $\widetilde S$ and $\tilde s \in \tilde p ^{\tilde \perp} \cap \widetilde R$. According to the result of \cite[Lemma 2.3]{on1}, it follows that $\tilde p^{\tilde \perp} \cap \widetilde R$ is a ``plane". Then, back in $\Gamma$, $p^{\perp} \cap R$ is a maximal singular subspace of $R$. Now $(p^{\perp} \cap R) \cap \overline B = \lbrace q \rbrace$ which implies $p^{\perp} \cap R \in M_{\mathcal A}(R)$. Thus $p \in N_{\mathcal A}(R)$.

\medskip
$(iv)$. Let now $p \in G_{\mathcal A}(S)$ be such that $p^{\perp} \cap {\overline B} = \emptyset$ and let $\lbrace q \rbrace = p^{\perp} \cap S$. Recall from Step $(i)$ that $q\in N_{\mathcal A}(R)$. Let $A_q \in {\mathcal A}(R) \cap {\mathcal A}_q$ be the maximal singular subspace which contains $\langle q, q^{\perp} \cap R \rangle$. Note that $A_q \in {\mathcal A}(S)$ also, since $A_q \cap S = \langle q, q^{\perp} \cap {\overline B}\rangle $. By Lemma $4.5(a), \; p^{\perp} \cap A_q$ is a plane which intersects $R$ nontrivially. Note that $p^{\perp} \cap A_q \cap R$ cannot contain a line, since this would imply $p^{\perp} \cap {\overline B} \not= \emptyset$, contrary to our choice of $p$. Let $\lbrace r \rbrace = p^{\perp} \cap A_q \cap R$. {\it Claim}: $p \in G_{\mathcal A}(R)$. We first prove that $p^{\perp} \cap R = \lbrace r \rbrace$ is a single point. Assume by contradiction that there is a point $s \in p^{\perp} \cap R \setminus \lbrace r \rbrace$. According to the above argument, $s \not \in A_q$. Then, by Lemma $4.4(a)$, applied to the pair $\lbrace p, q \rbrace$ and symplecton $R$ we get $p \in N_{\mathcal A}(R)$. But this implies $p^{\perp} \cap R \cap {\overline B} \not = \emptyset$, which contradicts the fact that $p^{\perp} \cap {\overline B} = \emptyset$. Therefore $p^{\perp} \cap R = \lbrace r \rbrace$. Since $p^{\perp} \cap A_q$ is a plane, it remains to prove that given any point $t \in r^{\perp} \cap R$, the pair $\lbrace p, t \rbrace$ is polar. This is clearly true for any $t \in r^{\perp} \cap {\overline B} = q^{\perp} \cap {\overline B}$, by the definition of $G_{\mathcal A}(S)$. So let us assume that $t \not \in {\overline B}$. Now $t \in N_{\mathcal A}(S)$ and since $q^{\perp} \cap t^{\perp} \cap S$ contains a plane, the pair $\lbrace q, t \rbrace$ is polar. Set $T = \ll q, t \gg$ and note that $T \cap S = \langle q, q^{\perp} \cap t^{\perp} \cap S\rangle  \in M_{\mathcal B}(S)$. Now $p \in G_{\mathcal A}(S)$, $p^{\perp} \cap S \cap T \not = \emptyset$ and $p^{\perp} \cap T \supset qr$ and according to Step $(iii)$, $p \in N_{\mathcal A}(T)$. Therefore $p^{\perp} \cap T \cap t^{\perp}$ contains a plane and this proves that $\lbrace p,t \rbrace$ is a polar pair.
\end{proof}

\begin{lem} a). Let $S, R \in {\mathcal S}$ be such that $S \cap R = L$ is a line. Assume that $R \cap {\mathcal A}(S) \not= \emptyset$. Then $D_{\mathcal A}(S) = D_{\mathcal A}(R)$.\\
b). Let $p \in G_{\mathcal A}(S)$ with $\lbrace q \rbrace = p^{\perp} \cap S$. For $r \in q^{\perp} \cap S$ set $R = \ll p, r \gg$. Then $D_{\mathcal A}(S) = D_{\mathcal A}(R)$.
\end{lem}

\begin{proof}a). Let $S, R \in {\mathcal S}$ be such that $S \cap R = L$ is a line. Let
$p \in R \cap G_{\mathcal A}(S)$. Denote $p ^{\perp} \cap S = \lbrace q
\rbrace$ and observe that $ q \in L$. Let $A \in {\mathcal A}(S)$ be such that $L \subset A$. Then, according to {\it (P2)}, $R \cap A$ can be a line or a maximal singular subspace of $R$. Since $p \in G_{\mathcal A}(S),\; p^{\perp} \cap A $ is a plane. Let $r \in p^{\perp} \cap A \setminus \lbrace q \rbrace$ and let $t \in L \setminus \lbrace q \rbrace$. Thus $R = \ll p, t \gg$ and $r \in R$. Then $A \cap R \supseteq \langle r, L \rangle $ and by {\it (P2)} it follows that $A \in {\mathcal A}(R)$. We proved that if $A \in {\mathcal A}(S)$ is such that $L \subset A$ then $A \in {\mathcal A}(R)$ as well.\\
Let now $A_1, A_2 \in {\mathcal A}(S) \cap {\mathcal A}(R)$ be two distinct maximal singular subspaces such that $L \subseteq A_1 \cap A_2$. The line $L$ is contained in the symplecton $S$ and so is the intersection of two maximal singular subspaces of $S$ of the same class. Let $p_1 \in A_1 \cap S \setminus L$ and $p_2 \in A_2 \cap R \setminus L$. Set $T = \ll p_1, p_2 \gg$. Note that $T \cap S = \langle p_1, p_1^{\perp} \cap A_2 \cap S \rangle $ and $T \cap R = \langle p_2, p_2^{\perp} \cap A_1 \cap R \rangle $ both singular subspaces of ${\mathcal B}$-type. Then, by Lemma $4.6$: $D_{\mathcal A}(S) = D_{\mathcal A}(T) = D_{\mathcal A}(R)$.

\medskip
b). Let $S$ and $R$ be two symplecta as in the hypothesis. We claim that $R \cap S = pq$, a line. If $R \cap S$ contains a plane then $p^{\perp} \cap S$ contains a line, contradicting the properties of $p \in G_{\mathcal A}(S)$. Since $p \in G_{\mathcal A}(S) \cap R$, part a) applies and the result follows.
\end{proof}

\begin{prop} Let $S \in {\mathcal S}$ then $D_{\mathcal A}(S)$ is a subspace of $\Gamma$.
\end{prop}

\begin{proof}Let $p, q \in D_{\mathcal A}(S)$ be two collinear points. We have to prove that $pq \subset D_{\mathcal A}(S)$.

\medskip
$(i)$. If $p,q \in S \cup N_{\mathcal A}(S)$ the result follows from Lemma $4.4(b)$.

\medskip
$(ii)$. If $p \in G_{\mathcal A}(S)$ and $ q \in S$ then $\lbrace q \rbrace = p^{\perp} \cap S$. In this case $pq \subset S \cup G_{\mathcal A}(S)$ follows from the gamma space property of $\Gamma$; see Section 2.5.

\medskip
$(iii)$. Assume now that $p \in N_{\mathcal A}(S)$ and $q \in G_{\mathcal A}(S)$.
Let ${\overline A}_p = p^{\perp} \cap S$ and $\lbrace r \rbrace = q^{\perp} \cap S$. It follows, from Lemma $4.4(a)$, that $r \in {\overline A}_p$. Let $t \in r^{\perp} \cap S \setminus {\overline A}_p$. Set $R = \ll p, t \gg$. Then, by Lemma $4.6$, $D_{\mathcal A}(S) = D_{\mathcal A}(R)$. Now $q \in G_{\mathcal A}(S) \subset D_{\mathcal A}(R)$ and $q^{\perp} \cap R$ contains the line $pr$, hence $q \in N_{\mathcal A}(R)$. Thus $p \in R,\; q \in N_{\mathcal A}(R)$ and, according to Lemma $4.4(b), \; pq \subset D_{\mathcal A}(R) = D_{\mathcal A}(S)$.

\medskip
$(iv)$. Let $p, q \in G_{\mathcal A}(S)$. First assume $p^{\perp} \cap
q^{\perp} \cap S = \lbrace r \rbrace$. Let $t \in r^{\perp} \cap S$ and set $R = \ll p, t\gg$. Then, according to Lemma $4.7(b)$, $D_{\mathcal A}(S) = D_{\mathcal A}(R)$. Thus $p \in R$ and $q \in R \cup N_{\mathcal A}(R)$ implies $pq \subset D_{\mathcal A}(R)= D_{\mathcal A}(S)$. Let now $p^{\perp} \cap S = \lbrace r \rbrace$ and $q^{\perp} \cap R = \lbrace t \rbrace$ be such that $r \not = t$. If $r \in t^{\perp}$ then set $R = \ll p, t \gg$ and, by Lemma $4.7(b)$, $D_{\mathcal A}(R) = D_{\mathcal A}(S)$. Therefore $pq \subset R \subset D_{\mathcal A}(S)$. If $r \not \in t^{\perp}$ then take $w \in r^{\perp} \cap t^{\perp}$ and set $T = \ll p, w \gg$. Again, by Lemma $4.7(b)$ $D_{\mathcal A}(T) = D_{\mathcal A}(S)$. So $p \in T,\; q \in D_{\mathcal A}(T)$ and by previous results of this Proposition, $pq \subset D_{\mathcal A}(T) = D_{\mathcal A}(S)$.
\end{proof}

The assumption {\it (WHA)} from the statement of Theorem $4.1$ is used in the proof of the following proposition.

\begin{prop} Let $S \in {\mathcal S}$, then $D_{\mathcal A}(S)$ is a $2$-convex
subspace of $\Gamma$.
\end{prop}

\begin{proof}Let $p, q \in D_{\mathcal A}(S)$ be two points at distance $2$. We have to prove that $p^{\perp} \cap q^{\perp} \subset D_{\mathcal A}(S)$.

\medskip
$(i)$. If $p,q \in S$ then obviously $p^{\perp} \cap q^{\perp} \subset
S$ since $S$ is a $2$-convex subspace of $\Gamma$.

\medskip
$(ii)$. Let now $p \in N_{\mathcal A}(S)$ and $q \in S$. The pair $\lbrace p,q \rbrace$ is polar. Set $R = \ll p,q \gg$. Since $R \cap S \in \mss{\mathcal B}{S}$ it follows, by Lemma $4.6$ that $D_{\mathcal A}(S) = D_{\mathcal A}(R)$. Hence $p^{\perp} \cap q^{\perp} \subset R \subset D_{\mathcal A}(S)$.

\medskip
$(iii)$. Consider now the case when $p,q \in N_{\mathcal A}(S)$. Then ${\overline A}_p = p^{\perp} \cap S$ and ${\overline A}_q =
q^{\perp} \cap S$ are two maximal singular subspaces of $S$ from the same
class. There are two cases to consider:

\smallskip
$(iii.a)$. Assume ${\overline A}_p \cap {\overline A}_q =L$ is a line. Let $r \in {\overline A}_q \setminus L$ and set $R = \ll r,p \gg$. Then $R \cap S \in M_{\mathcal B}(S)$ and, by Lemma $4.6,\; D_{\mathcal A}(S) = D_{\mathcal A}(R)$. So $q \in D_{\mathcal A}(R)$ and because $q^{\perp} \cap R \supseteq \langle r, L \rangle $ it follows $q \in N_{\mathcal A}(R)$. Note that $q \not \in R$ since $A_q \cap R = \langle r, p^{\perp} \cap A_q \rangle$ has rank $3$ and thus is already maximal in $R$. Now $p \in R,\; q \in N_{\mathcal A}(R)$ and according to Step $(ii)$ of this Proposition, $p^{\perp} \cap q^{\perp} \subset D_{\mathcal A}(R) =D_{\mathcal A}(S)$.

\smallskip
$(iii.b)$. Assume now that ${\overline A}_p$ and ${\overline A}_q$ are disjoint maximal singular subspaces of $S$. Let $r \in p^{\perp} \cap q^{\perp}$. Take $p_1 \in {\overline A}_p \setminus r^{\perp}$ and $q_1 \in {\overline A}_q \setminus ( r^{\perp} \cup p_1^{\perp})$. Also let $r_1 \in p_1^{\perp} \cap q_1^{\perp}$ be such that $r_1 \not \in p^{\perp} \cup q^{\perp}$. Now $H = (p_1, r_1, q_1, q,r,p)$ is a minimal $6$-circuit in $\Gamma$ which contains at least one polar pair $\lbrace p_1, q_1 \rbrace$. Then {\it (WHA)} applies and there exists a point $w \in r^{\perp} \cap p_1^{\perp} \cap q_1^{\perp} \subset S$. Thus, according to Lemma $4.4(a)$, $r \in N_{\mathcal A}(S) \subset D_{\mathcal A}(S)$.

\medskip
$(iv)$. Let $p \in G_{\mathcal A}(S)$ and $q \in S$. Let $\lbrace t \rbrace = p^{\perp} \cap S$. If $q \in t^{\perp}$ then $R = \ll p,q \gg$ is contained in $D_{\mathcal A}(S)$; see Lemma $4.7(b)$. So let us assume $q \not \in t^{\perp}$. Take $A \in {\mathcal A}(S) \cap {\mathcal A}_t$. Then according to Lemma $4.5$, $p^{\perp} \cap A$ is a plane. Let $p_1 \in p^{\perp} \cap A \setminus \lbrace t \rbrace$ and let $r \in t^{\perp} \cap S \setminus A$. Set $R = \ll p_1, r \gg$. Since $R \cap S = \langle r, r^{\perp} \cap A \rangle  \in \mss{\mathcal B}{R} \cap \mss{\mathcal B}{S}$ it follows, by Lemma $4.6$ that $D_{\mathcal A}(S) = D_{\mathcal A}(R)$. Now $p^{\perp} \cap R \supset p_1 t$ so $p \in N_{\mathcal A}(R)$. Also $q \in N_{\mathcal A}(R)$. According to Step $(iii)$ of this Proposition $p^{\perp} \cap q^{\perp} \subset D_{\mathcal A}(R) = D_{\mathcal A}(S)$.

\medskip
$(v)$. Let $p \in N_{\mathcal A}(S), \; q \in G_{\mathcal A}(S)$. Let
$\lbrace r \rbrace = q^{\perp} \cap S$ and ${\overline A}_p = p^{\perp} \cap S$. Assume first that $r \not\in {\overline A}_p$. Set $R = \ll p, r \gg$. Then $R \cap S
\in M_{\mathcal B}(S)$ and according to Lemma $4.6, \; D_{\mathcal A}(R) = D_{\mathcal A}(S)$. Now $p \in R,\; q \in D_{\mathcal A}(R)$ and using the above results of this Proposition, it follows $p^{\perp} \cap q^{\perp} \subset D_{\mathcal A}(R) = D_{\mathcal A}(S)$. Next consider the case $r \in {\overline A}_p$. Let $t \in {\overline A}_p$ be a point and set $T = \ll q,t \gg$. By Lemma $4.7(b),\; D_{\mathcal A}(S) = D_{\mathcal A}(T)$. So we may apply the results of the Steps $(i)-(iii)$ of this Proposition to the pair of points $ p \in D_{\mathcal A}(T),\; q \in T$ and conclude that $p^{\perp} \cap q^{\perp} \subset D_{\mathcal A}(T) = D_{\mathcal A}(S)$.

\medskip
$(vi)$. Let now $p,q \in G_{\mathcal A}(S)$ be such that $p^{\perp} \cap S =
\lbrace p_1 \rbrace$ and $q^{\perp} \cap S = \lbrace q_1 \rbrace$. If $p_1 \not \in q_1^{\perp}$ then let $r \in p_1^{\perp} \cap q_1^{\perp}$; if $p_1 \in q_1^{\perp} \setminus \lbrace q_1 \rbrace$ then we let $r = q_1$ and if $p_1 = q_1$ then take $r \in p_1^{\perp} \cap S$. Set $R =  \ll p, r \gg$. By Lemma $4.7(b),\; D_{\mathcal A}(S) = D_{\mathcal A}(R)$. Now $q \in G_{\mathcal A}(S) \subset D_{\mathcal A}(R)$ and $p \in R$ hence using previous results of this Proposition, $p^{\perp} \cap q^{\perp} \subset D_{\mathcal A}(R) = D_{\mathcal A}(S)$.
\end{proof}

\begin{lem} Let $p \in G_{\mathcal A}(S)$ with $\lbrace p_1 \rbrace = p^{\perp} \cap S$ and let $w \in S$. Then $d(p,w) = 1+d(p_1,w)$.
\end{lem}

\begin{proof} We start by proving the following {\it claim}: if $p,q \in G_{\mathcal A}(S)$ are two collinear points with $\lbrace p_1 \rbrace = p^{\perp} \cap S$ and $\lbrace q_1 \rbrace = q^{\perp} \cap S$ then $q_1 \in p_1^{\perp}$. Assume by contradiction that $q_1 \not \in p_1^{\perp}$. Let $A \in {\mathcal A}_{p_1} \cap{\mathcal A}(S)$. Then, according to Lemma $4.5,\; p^{\perp} \cap A$ is a plane. Let $x \in p^{\perp} \cap A \setminus \lbrace p_1 \rbrace$. Also let $r \in p_1^{\perp} \cap q_1^{\perp} \setminus A$. Set $R = \ll r, x \gg$ and observe that $R \cap S = \langle r, r^{\perp} \cap A \cap S\rangle  \in M_{\mathcal B}(S)$ and, by Lemma $4.6, \; D_{\mathcal A}(R) = D_{\mathcal A}(S)$. Note that $p^{\perp} \cap R \supset p_1x$ and consequently $p \in N_{\mathcal A}(R)$. Set $\overline A_p = p^{\perp} \cap R$. Also $q \in G_{\mathcal A}(S) \subset D_{\mathcal A}(R)$ implies that there is a point $y \in q^{\perp} \cap R$. If $y \not \in {\overline A}_p$, since $p$ and $q$ are collinear, it follows, by Lemma $4.4(a)$, that $q \in N_{\mathcal A}(R)$. But then $q^{\perp} \cap R$ and $R \cap S$ are maximal singular subspaces in $R$ from different families and therefore they have a point in common. Since $y \not \in R \cap S$ it follows that $q^{\perp} \cap R$ contains more than a point, which is a contradiction with the fact that $q \in G_{\mathcal A}(S)$. Therefore $y \in {\overline A}_p$. Next $q_1 \in S$ and, by Lemma $4.6$, $q_1 \in N_{\mathcal A}(R)$. Let $\overline A _{q_1} = q_1^{\perp} \cap R$ and observe that $y \in {\overline A}_{q_1}$ because otherwise another application of Lemma $4.4(a)$ to the pair $\lbrace q, q_1 \rbrace$ and symplecton $R$ would give $q \in N_{\mathcal A}(R)$, a contradiction. Therefore $y \in {\overline A}_p \cap {\overline A}_{q_1}$ which, using {\it (P3)} implies ${\overline A}_p \cap {\overline A}_{q_1} = L$ is a line. But $r \not \in x^{\perp}$ so $x \not \in L$. Since ${\overline A}_{q_1}$ has rank $3$ and since $S \cap R$ meets ${\overline A}_{q_1}$ in a plane it follows that $L \cap (R \cap S) \not = \emptyset$. But this implies that $p^{\perp} \cap S$ contains more than a point. We reach a contradiction with the fact that $p \in G_{\mathcal A}(S)$. Therefore the assumption made was false and $p_1 \in q_1^{\perp}$; the claim is proved.

\medskip
In order to prove the Lemma it suffices to show that, if $p \in G_{\mathcal A}(S)$ and $w \in S \setminus p_1^{\perp}$ then $d(p, w) = 3$. Assume by contradiction that $d(p, w) = 2$. Then there exists a point $t \in p^{\perp} \cap w^{\perp}$. Note that $t$ cannot be in $S$. According to Proposition $4.10,\; t \in D_{\mathcal A}(S)$. Moreover $t \not \in N_{\mathcal A}(S)$, because since $w \not \in p_1^{\perp}$, Lemma $4.4(a)$ would imply $p \in N_{\mathcal A}(S)$, a contradiction. So we must have $t \in G_{\mathcal A}(S)$. But now, according to the previous paragraph $w \in p_1^{\perp} \cap S$, which contradicts the hypothesis on $w$. Therefore the assumption made was false and in this case $d(p,w)=3$.
\end{proof}

\begin{prop} For any $S \in {\mathcal S},\ D_{\mathcal A}(S)$ is a strong parapolar
subspace of $\Gamma$.
\end{prop}

\begin{proof} We have to prove that if $p,q \in D_{\mathcal A}(S)$ are two points at distance two, the pair $\lbrace p,q \rbrace$ is polar, that is $p^{\perp} \cap q^{\perp}$ contains at least two points.

\medskip
$(i)$. Let $p \in D_{\mathcal A}(S)$ and $q \in S$. If both $p$ and $q$ are in $S$ then $S = \ll p, q \gg$ and we are done. Assume next $p \in N_{\mathcal A}(S)$. Then $q^{\perp} \cap p^{\perp} \cap S$ contains a plane and therefore $\lbrace p,q \rbrace$ is a polar pair. Let now $p \in G_{\mathcal A}(S)$ with $\lbrace p_1 \rbrace = p^{\perp} \cap S$. In this case, since $d(p,q) = 2$, by Lemma $4.10$ it follows that $q \in p_1^{\perp}$. Then the fact that $\lbrace p, q \rbrace$ is polar pair follows from the definition of $G_{\mathcal A}(S)$.

\medskip
$(ii)$. Assume now $p,q \in N_{\mathcal A}(S)$. Let ${\overline A}_p = p^{\perp} \cap S$ and ${\overline A}_q = q^{\perp} \cap S$. If ${\overline A}_p \cap {\overline A}_q \not = \emptyset$ the result is immediate. So we may assume that ${\overline A}_p \cap {\overline A}_q = \emptyset$. Let $r \in {\overline A}_p$. Then set $R = \ll q, r \gg$ which, by Lemma $4.6$, is such that $D_{\mathcal A}(R) = D_{\mathcal A}(S)$. Thus $p \in G_{\mathcal A}(S) \subset D_{\mathcal A}(R),\; q \in R$ and by Step $(i)$ of this Proposition it follows that $\lbrace p, q \rbrace$ is a polar pair.

\medskip
$(iii)$. Let $p \in N_{\mathcal A}(S)$ and $q \in G_{\mathcal A}(S)$. Let ${\overline A}_p = p^{\perp} \cap S$ and $\lbrace t \rbrace = q^{\perp} \cap S$. If $t \in {\overline A}_p$ then, according to Lemma $4.5(b),\; \lbrace p, q \rbrace$ is a polar pair. Assume next that $t \not \in {\overline A}_p$. Take a point $r \in t^{\perp} \cap {\overline A}_p$ and set $R = \ll q, r \gg$. By Lemma $4.7(b), \; D_{\mathcal A}(R) = D_{\mathcal A}(S)$. So $q \in R,\; p \in N_{\mathcal A}(S) \subset D_{\mathcal A}(R)$ and $\lbrace p,q \rbrace$ is a polar pair by Step $(i)$ of this Proposition.

\medskip
$(iv)$. Let now $p,q \in G_{\mathcal A}(S)$ with $\lbrace r \rbrace = p^{\perp} \cap S$ and $\lbrace t \rbrace = q^{\perp} \cap S$. If $r \in t^{\perp} \setminus \lbrace t \rbrace$ set $R = \ll r, q \gg$. If $r = t$ then take $w \in r^{\perp}$ and if $r \not \in t^{\perp}$ take $w \in r^{\perp} \cap t^{\perp}$ and set $R =\ll w, q \gg$. Then, by Lemma $4.7(b),\; D_{\mathcal A}(R) = D_{\mathcal A}(S)$ and since $p \in G_{\mathcal A}(S) \subset D_{\mathcal A}(R), q \in R$, the pair $\lbrace p,q \rbrace$ is polar by Step $(i)$ of this Proposition.
\end{proof}

\begin{prop}Given $R \in {\mathcal S}$ there exists a unique element
$D_{\mathcal A}(S) \in {\mathcal D}_{\mathcal A}$, for some $S \in {\mathcal S}$ containing $R$.
\end{prop}

\begin{proof}
Let $R$ be a symplecton and assume that $R \subset D_{\mathcal A}(S)$ for some $S \in {\mathcal S}$. Write $R=\ll p,q \gg$ with $p$ and $q$ two points of $R$, at distance two in the collinearity graph. We analyze the possible relations between $D_{\mathcal A}(S)$ and $D_{\mathcal A}(R)$.

\medskip
$(i)$. Assume $p \in D_{\mathcal A}(S)$ and $q \in S$. If $p \in S$ then $S= \ll p,q \gg = R$. Next let $p \in N_{\mathcal A}(S)$ and set ${\overline A}_p = p^{\perp} \cap S$. Then $R \cap S = \langle q, q^{\perp} \cap {\overline A}_p \rangle  \in M_{\mathcal B}(S)$. Hence $D_{\mathcal A}(S) = D_{\mathcal A}(R)$, by Lemma $4.6$. Let now $p \in G_{\mathcal A}(S)$ with $\lbrace r \rbrace = p^{\perp} \cap S$. By Lemma $4.11$, $q \in r^{\perp} \cap S$ and by Lemma $4.7(b)$ it follows that $D_{\mathcal A}(R) = D_{\mathcal A}(S)$.

\medskip
$(ii)$. Let $p, q \in N_{\mathcal A}(S)$ with ${\overline A}_p = p^{\perp} \cap S$ and ${\overline A}_q = q^{\perp} \cap S$.

\smallskip
$(ii.a)$. Assume ${\overline A}_p \cap {\overline A}_q = L$ is a line. {\it Claim}: $R \cap S = L$. Suppose by contradiction that $R \cap S$ contains a plane $\langle r,L \rangle $ where $r$ is a point not on $L$. Since $r \not \in A_p \cap A_q$ we can assume, without loss of generality, that $r \not \in A_q$. If $r \in {\overline A}_p$ then $\langle r, q^{\perp} \cap A_p, p \rangle$ is a singular subspace of rank $4$ in $R$, a contradiction. So let us assume that $r \in S \setminus ({\overline A}_p \cup {\overline A}_q)$. Then $\langle r, r^{\perp} \cap {\overline A}_p, r^{\perp} \cap {\overline A}_q \rangle  \subset R \cap S$ which is a contradiction with the fact that $R \cap S$ can be at most a common maximal singular subspace. Therefore the claim is proved. Now pick a point $z \in p^{\perp} \cap q^{\perp} \setminus L$ such that $z^{\perp} \cap L = \lbrace t \rbrace$, a single point. We intend to prove that $z \in G_{\mathcal A}(S)$. According to Proposition $4.11, \; z\in D_{\mathcal A}(S)$, so it suffices to prove that $z^{\perp} \cap S = \lbrace t \rbrace$. Assume by contradiction that $z \in N_{\mathcal A}(S)$ with ${\overline A}_z = z^{\perp} \cap S$. Then, according to {\it (P3)}, ${\overline A}_z \cap {\overline A}_p$ is a line $L'$ distinct from $L$. So if $w \in L \setminus \lbrace t \rbrace$ then $\langle L, L' \rangle \subset z^{\perp} \cap w^{\perp} \subset R \cap S$ and we reach a contradiction with the previous result that $R \cap S = L$. Therefore $z^{\perp} \cap S = \lbrace p \rbrace$ and $z \in G_{\mathcal A}(S)$. Now $R$ and $S$ are as in Lemma $4.7(a)$ and consequently $D_{\mathcal A}(R) = D_{\mathcal A}(S)$.

\smallskip
$(ii.b)$. Let ${\overline A}_p \cap {\overline A}_q = \emptyset$. Let $t \in {\overline A}_p$ and set $T = \ll t,q \gg$. Thus $T \cap S = \langle  t, t^{\perp} \cap {\overline A}_q \rangle  \in M_{\mathcal B}(S)$ and by Lemma $4.6,\; D_{\mathcal A}(T) = D_{\mathcal A}(S)$. Now $p \in D_{\mathcal A}(T), q \in T$ so by Step $1$, $R = \ll p, q \gg$ is such that $D_{\mathcal A}(R) = D_{\mathcal A}(T) = D_{\mathcal A}(S)$.

\medskip
$(iii)$. Let $p \in N_{\mathcal A}(S), \; q \in G_{\mathcal A}(S)$ with $q^{\perp} \cap S = \lbrace r \rbrace$. If $r \in {\overline A}_p$ then let another point $t \in {\overline A}_p$ and set $T = \ll q, t \gg$. By Lemma $4.7b$, $D_{\mathcal A}(T) = D_{\mathcal A}(S)$. Apply Step $(i)$ to $p \in D_{\mathcal A}(T),\; q \in T$ to get $D_{\mathcal A}(R) = D_{\mathcal A}(T) = D_{\mathcal A}(S)$. If $r \not \in {\overline A}_p$ set $T = \ll p, r \gg$ and since $T \cap S \in M_{\mathcal B}(S)$, Lemma $4.6$ gives $D_{\mathcal A}(T)=D_{\mathcal A}(S)$. Now $p \in T$ and $q \in D_{\mathcal A}(T)$ and using Step $(i)$ again we get $D_{\mathcal A}(R) = D_{\mathcal A}(T) = D_{\mathcal A}(S)$.

\medskip
$(iv)$. Let $p,q \in G_{\mathcal A}(S)$ with $\lbrace p_1 \rbrace =
p^{\perp} \cap S$ and $\lbrace q_1 \rbrace = q^{\perp} \cap S$. If $p_1 = q_1$ take $r \in p_1^{\perp} \cap S$, set $T = \ll r, p \gg$. If $p_1 \in q_1^{\perp} \setminus \lbrace q_1 \rbrace$ take $T = \ll p, q_1 \gg$. If $p_1 \not \in q_1^{\perp}$ then take $r \in p_1^{\perp} \cap q_1^{\perp}$ and set $T =  \ll p,r \gg$. Then apply Step $(i)$ to get $D_{\mathcal A}(R) = D_{\mathcal A}(T) = D_{\mathcal A}(S)$.
\end{proof}

\medskip
\begin{proof}[{\it Proof Theorem 4.1}] Let {\pl} be a parapolar space
which is locally $A_{7,4}$. Assume that the maximal singular subspaces of $\Gamma$ partition in two classes. In addition, assume that $\Gamma$ satisfies the Weak Hexagon Axiom {\it (WHA)}. For any symplecton $S \in {\mathcal S}$ we defined two sets $D_{\mathcal X}(S)$ with ${\mathcal X} \in \lbrace {\mathcal A}, {\mathcal B} \rbrace$; see the beginning of Section $4$. Then $D_{\mathcal A}(S)$ is a $2$-convex subspace of
$\Gamma$, see Propositions $4.8$ and $4.9$, which is strong parapolar, by
Proposition $4.11$. Further, by Proposition $4.12$, every symplecton lies in a unique element of ${\mathcal D}_{\mathcal A}$.

\medskip
Consider a point-symplecton pair $(p, R)$ in $D_{\mathcal A}(S)$. Since $R \subset D_{\mathcal A}(S)$ it follows, by Proposition $4.12$ that $D_{\mathcal A}(R)=D_{\mathcal A}(S)$. Therefore, either $p \in R$ or $p^{\perp} \cap R$ is a single point or a maximal singular subspace of $R$. By the characterization theorem of Cohen and
Cooperstein \cite{cc}, see section $2.4$ also, $D_{\mathcal A}(S)$ is a space of type $D_{6,6}$.

\medskip
Next, replace ${\mathcal A}$ with ${\mathcal B}$. All of the above arguments can be repeated and we can conclude that $D_{\mathcal B}(S)$ is a subspace of $\Gamma$ which is of type $D_{6,6}$. Also, every symplecton $S \in {\mathcal S}$ is contained in exactly one element from the family ${\mathcal D}_{\mathcal B}$.
\end{proof}

\section{The sheaf theoretic characterization}

In this Section, we combine Ronan-Brouwer-Cohen sheaf theory with Tits' Local Approach Theorem and prove the following:

\begin{thm}Let ${\widetilde \Gamma}= ( {\widetilde {\mathcal P}}, {\widetilde {\mathcal L}})$ be a parapolar space which is locally of type $A_{7, 4}$. Let $I = \lbrace 1, \ldots 8 \rbrace, \; J = \lbrace 1, 8 \rbrace$ and $K = I \setminus J$. Assume that ${\widetilde \Gamma}$ satisfies the Weak Hexagon Axiom. Then there is a residually connected $J$-locally truncated diagram geometry $\Gamma$ which belongs to the diagram:\\
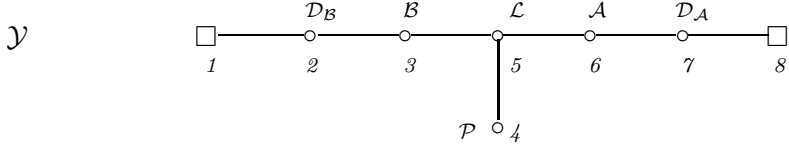
\begin{picture}(1000, 50)(0,0)
\put(0,25){$\mathcal{Y}$}
\put(60,25){$\qed$}
\put(75,15){\scriptsize 1}
\put(80,29){\line(1,0){32}}
\put(112,26){$\circ$}
\put(113,15){\scriptsize 2}
\put(113,36){\scriptsize ${\mathcal D}_{\mathcal B}$}
\put(118,29){\line(1,0){30}}
\put(148,26){$\circ$}
\put(150,36){\scriptsize ${\mathcal B}$}
\put(150,15){\scriptsize 3}
\put(153,29){\line(1,0){30}}
\put(183,26){$\circ$}
\put(190,36){\scriptsize ${\mathcal L}$}
\put(190,15){\scriptsize 5}
\put(186,-3){\line(0,1){30}}
\put(183,-9){$\circ$}
\put(171,-10){\scriptsize ${\mathcal P}$}
\put(190,-10){\scriptsize 4}
\put(188,29){\line(1,0){30}}
\put(218,26){$\circ$}
\put(220,36){\scriptsize ${\mathcal A}$}
\put(220,15){\scriptsize 6}
\put(223,29){\line(1,0){30}}
\put(253,26){$\circ$}
\put(253,36){\scriptsize ${\mathcal D}_{\mathcal A}$}
\put(255,15){\scriptsize 7}
\put(258,29){\line(1,0){30}}
\put(276,25){$\qed$}
\put(290,15){\scriptsize 8}
\end{picture}

\vspace{.4cm}
whose universal $2$-cover is the truncation of a building. Therefore, ${\widetilde \Gamma}$ is the homomorphic image of a truncated building, also.
\end{thm}

\begin{proof}Let us assume that ${\widetilde \Gamma}= ( {\widetilde {\mathcal P}}, {\widetilde {\mathcal L}})$ is a parapolar space, locally of type $A_{7,4}$ and which satisfies {\it(WHA)}. Construct the space {\pl} according to the method described in Section $3$. Then $\Gamma$ is also locally of type $A_{7,4}$ and satisfies {\it(WHA)}. Furthermore, $\Gamma$ has two classes of maximal singular subspaces ${\mathcal A}$ and ${\mathcal B}$, whose elements are projective spaces of rank $5$. It was proved in Section $4$ that in $\Gamma$ there are two more families of subspaces, denoted ${\mathcal D}_{\mathcal A}$ and ${\mathcal D}_{\mathcal A}$, whose elements are of type $D_{6,6}$. Define incidence as follows:
\begin{itemize}
\item[(i).]An object in $\mathcal{A}$ is incident with an object in $\mathcal{B}$ if they intersect at a plane.
\item[(ii).] An object in $\mathcal{B}$ (or $\mathcal{A}$) is incident with an object from $\mathcal{D}_{\mathcal A}$ ($\mathcal{D}_{\mathcal B}$ respectively) if they intersect at a singular subspace of rank $3$.
\item[(iii).] An object from $\mathcal{D}_{\mathcal A}$ is incident with an object from $\mathcal{D}_{\mathcal B}$ if they have a symplecton in common.
\item[(iv).] Inclusion for all the remaining cases.
\end{itemize}

\medskip
Hence, the space $\Gamma$ can be enriched to a rank $6$ geometry $\Gamma = (\mathcal{P, L, A, B}, {\mathcal D}_{\mathcal A}, {\mathcal D}_{\mathcal B})$. It follows from Theorem $3.1$ that $\Gamma$ is a covering of ${\widetilde \Gamma}$. Furthermore, $\Gamma$ is a geometry over $K$ which is $J$-locally truncated with respect to the diagram $\mathcal{Y}$.

\medskip
In the sequel we will use the terminology given in Sections $2.1$-$2.3$; the definition of a sheaf over a locally truncated geometry was given in Section $2.4$.

\medskip
Next we construct a sheaf $\Sigma$ over the collection $\mathcal{F}$ of flags of rank $1$ and $2$ of $\Gamma$. It is a consequence of a result of Ellard and Shult \cite{els}, see also \cite[Section 6]{on1}, that it suffices to work with the collection $\mathcal{F}$ of flags, instead of the full family of nonempty flags of $\Gamma$.

\medskip
${\mathfrak 1})$. For $a \in {}^{2}\Gamma$, an object of type $2$ in $\Gamma$, define $\Sigma ^R(a)$ to be the geometry of $\lbrace 8 \rbrace$-locally truncated type belonging to the diagram $\mathcal{Y}_{I \setminus \lbrace 1,2 \rbrace}$ and which satisfies the property that $\res {\Gamma}{a} =\;  _{\lbrace 8 \rbrace}\Sigma ^R(a)$. This is well defined since $\Sigma ^R (a)$ is of $\lbrace 8 \rbrace$-locally truncated type $D_{6,6}$, which, up to the relabeling of the nodes is unique. In $\Sigma
^R (a)$ the objects of types in $K$ are the same as in $\Gamma$ with the corresponding incidence. The objects of type $\lbrace 8 \rbrace$ are collections of objects in $\Gamma$ with their flags, which are incident with a given object of type in $K$; the incidence between objects of type $\lbrace 8 \rbrace$ is given by symmetrized containment.

\medskip
${\mathfrak 2})$. For $b \in \; ^{4}\Gamma$ define $\Sigma(b)$ to be such that $\res{\Gamma}{ b} =\;_J \Sigma(b)$  Observe that $\res{\Gamma}{b}$ is
the $J$-truncation of a geometry belonging to the diagram $D_{I
\setminus \lbrace 4 \rbrace}$ of type $A_{7,4}$. But this is
uniquely determined by its truncation, \cite[Theorem 1]{loc}, and thus $\Sigma(b)$ is unambiguously defined. There are two types of objects in $\Sigma
(b)$: those inherited from $\Gamma$, with their incidence and the objects with types
in $J$, which are defined as in ${\mathfrak 1})$.

\medskip
${\mathfrak 3})$. Let now $l \in \lbrace 3, \dots, 7 \rbrace$ and let $x_l \in \;
^{l}\Gamma$. Denote by $F = \lbrace a,x \rbrace $ a $\lbrace 2, l-1 \rbrace $-flag in $\res{\Gamma}{x_l}$. For $l=3$, $F = \lbrace a \rbrace$ is just an object of type $2$. Define recursively: $\Sigma ^R (x_l) = \Sigma ^R (a, x_l) := \res{\Sigma ^R (a,x)}{x_l}$. Let $F'$ be another flag of type
$\lbrace 2, l-1 \rbrace$ in $\res{\Gamma }{x_l}$ which is $i$-adjacent to $F$ in $C$, the chamber system associated to $^{\lbrace 2, l-1 \rbrace}\res {\Gamma} {x_l}$ with  $i \in \lbrace 2, l-1 \rbrace$. Then: $\res{\Sigma ^R (F)}{x_l} = \res{\Sigma ^R (F \cap F')} {(F \setminus F') \cup \lbrace x_l \rbrace } = \res{\Sigma ^R (F') }{x_l}$ which proves the well-definedness in this case. If $F$ and $F'$ are not
$i$-adjacent, since the chamber system $C$ is connected, there is a
gallery from $F$ to $F'$, and by repeated applications of the above
argument we get the result.

\medskip
${\mathfrak 4})$. Let $f \in \; ^{7} \Gamma$. Define $\Sigma ^L (f)$ to be the
geometry of $\lbrace 1 \rbrace$-locally truncated type, with diagram $\mathcal{Y}_{I \setminus \lbrace 7,8 \rbrace}$, and such that $\res{\Gamma}{ f} = \;
_{\{1\}} \Sigma ^L (f)$; this is uniquely defined since it has diagram of type $D_{6,6}$. The objects in $\Sigma ^L(f)$ and their incidence are defined as in the first step.

\medskip
${\mathfrak 5})$. Let $l \in \lbrace 6, \ldots 2 \}$, taken in this order. Let $x_l \in
\; ^{l}\Gamma$ and $F = \lbrace x, f \rbrace$ a $\lbrace l+1, 7 \rbrace$-flag in $\res{\Gamma}{x_l}$. Recursively define: $\Sigma ^L (x_l) = \Sigma ^L (x_l, f) := \res{\Sigma ^L(x, f)} {x_l}$. The proof of well-definedness is similar to the one given in the third step. The objects in $\Sigma ^L(x_l)$ and their incidence are defined using the method from the first step.

\medskip
The sheaf values over the rank $1$ flags of $\Gamma$ can be written as
follows:
\begin{itemize}
\item[ (i).] \quad $\Sigma (x) := \Sigma ^L (x) \oplus \Sigma ^R (x),\;   \text{for
any object}\ x\  \text{of type in}\  K \setminus \lbrace 4 \rbrace$;
\item[(ii).] \quad $\Sigma (b) \; \text{for any object}\ b \in \; ^{4}\Gamma$.
\end{itemize}

Next, we construct the sheaf the sheaf over the rank $2$ flags of $\Gamma$ and we check the compatibility condition: $\res{\Sigma(x)}{y} = \Sigma(x,y) = \res{\Sigma(y)}{x}$, where $\lbrace x, y \rbrace$ is a nonempty rank $2$ flag in $\Gamma$.

\medskip
${\mathfrak 6})$. For $\lbrace x_i, x_j \rbrace$, a rank $2$ flag of type $\lbrace i, j\rbrace$ in $\Gamma$ with $i, j \in K \setminus \lbrace 4 \rbrace$
and $i<j$, set
$$\Sigma (x_i, x_j) := \Sigma ^{L} (x_i) \oplus \;
^{(ij)} \Gamma \oplus \Sigma ^{R}(x_j)$$
where $^{(ij)}\Gamma$ denotes the
truncation of $\Gamma$ to those objects of type $l \in K$ with $i<l<j$. Since
\begin{equation*}
\begin{split}
&\res{\Sigma(x_i)}{x_j} = \res {\Sigma ^L (x_i)}{x_j} \oplus \res {\Sigma ^R(x_i)}{x_j} = \Sigma ^L (x_i) \oplus \; ^{(ij)}\Gamma \oplus \Sigma ^R (x_j)\\
&= \res {\Sigma ^L (x_j)}{x_i} \oplus \Sigma ^R (x_j) = \res{\Sigma (x_j)}{x_i}
\end{split}
\end{equation*}
it follows that the compatibility condition is satisfied in this case.

\medskip
${\mathfrak 7})$. There are two remaining cases, as follows:
\begin{itemize}
\item[a).] For $\lbrace a,b \rbrace$, a $\{ 3,4\}$-flag, define $\Sigma (a,b) := \Sigma ^L(a) \oplus \Sigma ^R (b)$.
\item[b).] For a flag $ \lbrace b, x_j \rbrace$ of type $\{4, j \}$, with $j \in \{5,6,7\}$ set $\Sigma(b, x_j) := \res{\Sigma ^L(x_j)}{b} \oplus \Sigma ^R(x_j)$.
\end{itemize}
The compatibility condition can be checked now, the proof is similar to the one given in the previous step.

\medskip
Therefore $\Sigma$ is defined for all rank $1$ and rank $2$ flags of
$\Gamma$. Now using the aforementioned Lemma, see \cite[Section 6]{on1}, of Ellard and Shult, we can extend the sheaf $\Sigma$ to a sheaf (denoted the same) over all nonempty flags of $\Gamma$.

\medskip
The sheaf $\Sigma$ is residually connected; this follows from the fact that all sheaf values at rank $1$ and $2$ flags of $\Gamma$ are $J$-truncated buildings or products of $J$-truncated buildings.

\medskip
According to a result of Brouwer and Cohen \cite{loc}, there exists a canonically defined chamber system $\mathcal{C}(\Sigma)$ over $I$, which is residually connected and belongs to diagram $\mathcal{Y}$. It satisfies the property that $_J {\mathcal C}(\Sigma) \; \simeq \; {\mathcal C}(\Gamma)$, an isomorphism of chamber systems over $K$, with the term on the right being the chamber system corresponding to the geometry $\Gamma$. Moreover, all rank $3$ residues are covered by truncations of buildings. Therefore by Tits' Local Approach Theorem \cite{tits}, see also Theorem $2.1$,  the universal $2$-cover of ${\mathcal C}(\Sigma)$ is the chamber system ${\mathcal B}$ of a truncated building belonging to diagram $\mathcal{Y}$.

\medskip
There is a $J$-locally truncated building geometry $\widetilde \Delta $ which belongs to the diagram $\mathcal{Y}$ and which is the image under the functor ${\mathbf G}: {\mathcal
Chamb}_I \rightarrow {\mathcal Geom}_I$ of $\mathcal{B}$; see Section 2.3. The chamber system ${\mathcal C}(\Sigma)$ also has a corresponding geometry which we shall denote $\Delta$. Since the chamber system ${\mathcal C}(\Sigma)$ is
residually connected, it follows that $\Delta$ is also residually connected. The $2$-cover map of chamber systems ${\mathcal B} \rightarrow \mathcal{C}(\Sigma)$ functorially induces a covering of geometries $\widetilde{\Delta} \rightarrow \Delta$. The fact that $\widetilde{\Delta}$ is the universal cover of $\Delta$ follows from the fact that $\mathcal{B}$ is the universal cover of $\mathcal{C}(\Sigma)$. Moreover the geometry $\Delta$ has the property that $ _J \Delta \simeq \Gamma$; for details see \cite{on1}. Combining the previous results with Theorem $3.1$, it follows that $\widetilde{\Gamma}$ is the homomorphic image of a building.
\end{proof}

\end{document}